\documentclass{amsart}
\usepackage[all]{xy}
\usepackage{color}
\usepackage{amsthm}
\usepackage{amssymb}
\usepackage[colorlinks=true]{hyperref}
\usepackage[usenames,dvipsnames]{xcolor}

\numberwithin{equation}{section}

\newtheorem{theorem}[equation]{Theorem}
\newtheorem*{theorem*}{Theorem}
\newtheorem{lemma}[equation]{Lemma}

\newtheorem*{conjecture*}{Mamma Conjecture}
\newtheorem*{conjecture1*}{Mamma Conjecture (revisited)}
\newtheorem{proposition}[equation]{Proposition}
\newtheorem{corollary}[equation]{Corollary}
\newtheorem*{corollary*}{Corollary}

\theoremstyle{remark}
\newtheorem{definition}[equation]{Definition}

\newtheorem{example}[equation]{Example}

\newtheorem{notation}[equation]{Notation}

\theoremstyle{remark}
\newtheorem{remark}[equation]{Remark}

\newcommand{\cA}{{\mathcal A}}
\newcommand{\cB}{{\mathcal B}}
\newcommand{\cC}{{\mathcal C}}
\newcommand{\cD}{{\mathcal D}}
\newcommand{\cE}{{\mathcal E}}

\newcommand{\cO}{{\mathcal O}}

\newcommand{\cS}{{\mathcal S}}

\newcommand{\bbC}{\mathbb{C}}

\newcommand{\bbP}{\mathbb{P}}

\newcommand{\bbQ}{\mathbb{Q}}
\newcommand{\bbZ}{\mathbb{Z}}

\DeclareMathOperator{\SmProj}{SmProj} 

\DeclareMathOperator{\id}{id}

\DeclareMathOperator{\dgcat}{dgcat} 

\DeclareMathOperator{\Fun}{Fun} 

\newcommand{\perf}{\mathrm{perf}}

\newcommand{\Chow}{\mathsf{Chow}}

\newcommand{\dg}{\mathsf{dg}}

\newcommand{\Hom}{\mathrm{Hom}}
\newcommand{\End}{\mathrm{End}}

\newcommand{\KM}{\mathsf{KM}}

\newcommand{\Hmo}{\mathsf{Hmo}}
\newcommand{\op}{\mathrm{op}}

\newcommand{\too}{\longrightarrow}

\newcommand{\add}{\mathsf{add}}

\newcommand{\ie}{\textsl{i.e.}\ }
\newcommand{\eg}{\textsl{e.g.}}

\begin{document}

\title[Relations between the Chow and the NC motive of a variety]{Relations between the Chow motive \\and the noncommutative motive \\of a smooth projective variety}
\author{Marcello Bernardara and Gon{\c c}alo~Tabuada}

\address{Institut de Math\'ematiques de Toulouse \\ %
Universit\'e Paul Sabatier \\ %
118 route de Narbonne \\ %
31062 Toulouse Cedex 9\\ %
France}
\email{marcello.bernardara@math.univ-toulouse.fr} 
\urladdr{http://www.math.univ-toulouse.fr/~mbernard/}

\address{Gon{\c c}alo Tabuada, Department of Mathematics, MIT, Cambridge, MA 02139, USA}
\email{tabuada@math.mit.edu}
\urladdr{http://math.mit.edu/~tabuada}
\thanks{G. Tabuada was partially supported by a  NSF CAREER Award.}

\date{\today}


\abstract{In this note we relate the notions of Lefschetz type, decomposability, and isomorphism, on Chow motives with the notions of unit type, decomposability, and isomorphism, on noncommutative motives. Examples, counter-examples, and applications are also described.}}

\maketitle
\vskip-\baselineskip
\vskip-\baselineskip



\section{Introduction}
Let $k$ be a base field and $R$ a commutative ring of coefficients.
\subsection*{Chow motives}
In the early sixties Grothendieck envisioned the existence of a ``universal'' cohomology theory of schemes. Among several conjectures and developments, a contravariant $\otimes$-functor
$$ M(-)_R: \SmProj(k)^\op \too \Chow(k)_R$$
from smooth projective $k$-schemes to {\em Chow motives} (with $R$ coefficients) was constructed. Intuitively speaking, $\Chow(k)_R$ encodes all the geometric/arithmetic information about smooth projective $k$-schemes and acts as a gateway between algebraic geometry and the assortment of the numerous Weil cohomology theories such as de Rham, Betti, $l$-adic, crystalline, etc; see~\cite{Andre,Jannsen,Manin}.
\subsection*{Noncommutative motives}
A {\em differential graded (=dg) category} $\cA$ is a category enriched over complexes of $k$-vector spaces; see \S\ref{sub:dg}. Every (dg) $k$-algebra $A$ gives naturally rise to a dg category $\underline{A}$ with a single object and (dg) $k$-algebra of endomorphisms $A$. Another source of examples is provided by $k$-schemes since the category of perfect complexes $\perf(X)$ of every smooth projective $k$-scheme $X$ admits a unique dg enhancement $\perf_\dg(X)$; see \cite{LO}. All the classical invariants such as algebraic $K$-theory, cyclic homology, and topological Hochschild homology, extend naturally from $k$-algebras (and from $k$-schemes) to dg categories. In order to study all these invariants simultaneously the notion of additive invariant was introduced in \cite{IMRN}. Roughly speaking, a functor $E:\dgcat(k)\to \mathsf{D}$ from the category of dg categories towards an additive category is called {\em additive} if it inverts Morita equivalences and sends semi-orthogonal decompositions to direct 
sums. Thanks to the work \cite{BM,Exact,Exact2,Negative,Fundamental, MacLane,TT,Wald}, all the above mentioned invariants are additive. In \cite{IMRN} the universal 
additive invariant was also constructed
\begin{equation}\label{eq:universal}
U(-)_R:\dgcat(k) \too \Hmo_0(k)_R\,.
\end{equation}
Given any $R$-linear additive category $\mathsf{D}$, there is an induced equivalence of categories
\begin{equation}\label{eq:categories}
U(-)_R^\ast: \Fun_{\add}(\Hmo_0(k)_R,\mathsf{D}) \stackrel{\sim}{\too} \Fun_{\mathsf{A}}(\dgcat(k),\mathsf{D})\,,
\end{equation}
where the left-hand side denotes the category of additive functors and the right-hand side the category of additive invariants. Because of this universal property, which is reminiscent from motives, $\Hmo_0(k)_R$ is called the category of {\em noncommutative motives}. The tensor product of $k$-algebras extends also naturally to dg categories giving rise to a symmetric monoidal structure $-\otimes-$ on $\dgcat(k)$ which descends to $\Hmo_0(k)_R$ making the above functor \eqref{eq:universal} symmetric monoidal.
\subsection*{Motivating questions}
Let $X$ be a smooth projective $k$-scheme. In order to study it we can proceed in two distinct directions. On one direction we can associate to $X$ its Chow motive $M(X)_R$. On another direction we can associate to $X$ the noncommutative motive $U(\perf_\dg(X))_R$. Note that while $M(X)_R$ encodes all the information about the numerous Weil cohomology theories of $X$, $U(\perf_\dg(X))_R$ encodes all the information about the different additive invariants of $\perf_\dg(X)$.

Let ${\bf L}\in \Chow(k)_R$ be the Lefschetz motive and ${\bf 1}:=U(\underline{k})_R$ the $\otimes$-unit of $\Hmo_0(k)_R$. Following \cite{GO}, a Chow motive is called of {\em Lefschetz $R$-type} if it is isomorphic to
${\bf L}^{\otimes l_1} \oplus \cdots \oplus {\bf L}^{\otimes l_m}$ for some choice of non-negative integers $l_1, \ldots, l_m$. In the same vein, a noncommutative motive in $\Hmo_0(k)_R$ is called of {\em unit $R$-type} if it is isomorphic to $\oplus_{i=1}^m{\bf 1}$ for a certain non-negative integer $m$. The following implication was proved in \cite[\S4]{GO} (assuming that $\bbZ \subseteq R$): 
\begin{equation}\label{eq:implication-1}
M(X)_R\,\,\mathrm{Lefschtez}\,\,R\text{-}\mathrm{type} \Rightarrow U(\perf_\dg(X))_R\,\, \mathrm{unit} \,\,R\text{-}\mathrm{type}\,.
\end{equation}
In the particular case where $R=\bbQ$, this implication becomes an equivalence
\begin{equation}\label{eq:equivalence}
M(X)_\bbQ\,\,\mathrm{Lefschetz}\,\,\bbQ\text{-}\mathrm{type} \Leftrightarrow U(\perf_\dg(X))_\bbQ\,\,\mathrm{unit}\,\,\bbQ\text{-}\mathrm{type}\,;
\end{equation}
see \cite[\S1]{MT}. Hence, it is natural to ask the following:

\smallbreak

{\it Question A: Does the above implication \eqref{eq:implication-1} admits a partial converse ?}

\smallbreak

Recall that an object in an additive category is called indecomposable if its only non-trivial idempotent endomorphisms are $\pm$ the identity; otherwise it is called decomposable. Our second motivating question is the following:

\smallbreak

{\it Question B: What is the relation between the (in)decomposability of $M(X)_R$ and the (in)decomposability of $U(\perf_\dg(X))_R$ ?}

\smallbreak

Let $X$ and $Y$ be smooth projective $k$-schemes. Another motivating question is:

\smallbreak

{\it Question C: Does the following implication holds
\begin{equation}\label{eq:implication-new}
M(X)_R \simeq M(Y)_R \Rightarrow U(\perf_\dg(X))_R \simeq U(\perf_\dg(Y))_R\,\,\, ?
\end{equation}
How about its converse ?
}

\smallbreak

In this note we provide precise answers to these questions; consult Corollaries \ref{cor:main} and \ref{cor:main2} for applications.
\section{Statements of results}
Our first main result, which answers Question A, is the following:
\begin{theorem}\label{thm:main1}
Let $X$ be an irreducible smooth projective $k$-scheme of dimension $d$. Assume that $\bbZ\subseteq R$ and that every finitely generated projective $R[1/(2d)!]$-module is free (\eg\ $R$ a principal ideal domain). Assume also that $U(\perf_\dg(X))_R \simeq \oplus_{i=1}^m{\bf 1}$ for a certain non-negative integer $m$. Under these assumptions, there is a choice of integers (up to permutation) $l_1, \ldots, l_m \in \{0, \ldots, d\}$ giving rise to an isomorphism
$$ M(X)_{R[1/(2d)!]} \simeq {\bf L}^{\otimes l_1} \oplus \cdots \oplus {\bf L}^{\otimes l_m}\,.$$
\end{theorem}
Intuitively speaking, Theorem~\ref{thm:main1} shows that the converse of the above implication \eqref{eq:implication-1} holds as soon as one inverts the integer $(2d)!$ (or equivalently its prime factors). By combining this result with \eqref{eq:implication-1}, one obtains a refinement of \eqref{eq:equivalence}:
\begin{corollary}\label{cor:main1}
Given $X$ and $R$ as in Theorem~\ref{thm:main1}, we have the equivalence
\begin{equation*}
M(X)_{R[1/(2d)!]}\,\,\mathrm{Lefschetz}\text{-}\mathrm{type} \Leftrightarrow U(\perf_\dg(X))_{R[1/(2d)!]}\,\,\mathrm{unit}\text{-}\mathrm{type}\,.
\end{equation*}
\end{corollary}
In the particular case where $X$ is a curve $C$ and $R=\bbZ$, Corollary~\ref{cor:main1} reduces to 
\begin{equation*}
M(C)_{\bbZ[1/2]}\,\,\mathrm{Lefschetz}\,\,\bbZ[1/2]\text{-}\mathrm{type} \Leftrightarrow U(\perf_\dg(C))_{\bbZ[1/2]}\,\,\mathrm{unit}\,\,\bbZ[1/2]\text{-}\mathrm{type}\,.
\end{equation*}
Moreover, since the prime factors of $4!$ are $\{2,3\}$, one has 
\begin{equation*}
M(S)_{\bbZ[1/6]}\,\,\mathrm{Lefschetz}\,\,\bbZ[1/6]\text{-}\mathrm{type} \Leftrightarrow U(\perf_\dg(S))_{\bbZ[1/6]}\,\,\mathrm{unit}\,\,\bbZ[1/6]\text{-}\mathrm{type}
\end{equation*}
for every surface $S$. As the following proposition shows, the (strict) converse of implication \eqref{eq:implication-1} is false !
\begin{proposition}{(see \S\ref{sec:quadric})}\label{prop:counter-example-quadric}
Let $q$ be a non-singular quadratic form and $Q_q$ the associated smooth projective quadric. Assume that $q$ is even dimensional, anisotropic, and with trivial discriminant and trivial Clifford invariant. Under these assumptions, $M(Q_q)_\bbZ$ is {\em not} of Lefschetz $\bbZ$-type
while $U(\perf_\dg(Q_q))_\bbZ$ is of unit $\bbZ$-type (and hence of unit $R$-type for every commutative ring $R$).
\end{proposition}
Proposition \ref{prop:counter-example-quadric} applies to all $3$-fold Pfister forms and to all elements of the third power of the fundamental ideal $I(k)$ of the Witt ring $W(k)$; see Example~\ref{ex:3-fold}. 

As an application of Theorem~\ref{thm:main1}, we obtain the following sharpening of the main result of \cite{MT}; recall from {\em loc. cit.} that the isomorphism \eqref{eq:canonical-iso} below was obtained only with rational coefficients.
\begin{corollary}\label{cor:main}
Let $X$ be an irreducible smooth projective $k$-scheme of dimension $d$. Assume that $\perf(X)$ admits a full exceptional collection $(\cE_1,\ldots, \cE_m)$ of length $m$. Under these assumptions, there is a choice of integers (up to permutation) $l_1, \ldots, l_m \in \{0, \ldots,d\}$ giving rise to an isomorphism 
\begin{equation}\label{eq:canonical-iso}
M(X)_{\bbZ[1/(2d)!]} \simeq {\bf L}^{\otimes l_1} \oplus \cdots \oplus {\bf L}^{\otimes l_m}\,.
\end{equation}
\end{corollary}
Thanks to the work \cite{Beilinson,kapranovquadric,Kawamata,kuznetfanothreefolds,orlovprojbund}, Corollary~\ref{cor:main} applies to projective spaces and rational surfaces (in the case of an arbitrary base field $k$), and to smooth quadric hypersurfaces, Grassmannians, flag varieties, Fano threefolds with vanishing odd cohomology, and toric varieties (in the case where $k=\bbC$). Conjecturally, it applies also to all the homogeneous spaces of the form $G/P$, with $P$ a parabolic subgroup of a semisimple algebraic group $G$; see \cite{KP}.

Our second main result, which partially answers Question B, is the following:
\begin{theorem}\label{thm:main2}
Let $X$ be an irreducible smooth projective $k$-scheme of dimension $d$. Under the assumption $\bbZ\subseteq R$, the following implication holds:
\begin{equation}\label{eq:implication-indecomp}
M(X)_{R[1/(2d)!]}\,\,\mathrm{decomposable} \Rightarrow U(\perf_\dg(X))_{R[1/(2d)!]}\,\,\mathrm{decomposable}\,.
\end{equation}
\end{theorem}
It is unclear to the authors if the (strict) converse of \eqref{eq:implication-indecomp} also holds. As the following proposition shows, if one does not invert the dimension of $X$, this is false~!
\begin{proposition}\label{prop:main3}
Let $A$ be a central simple $k$-algebra of degree $\sqrt{\mathrm{dim}(A)}=d$ and $X=\mathrm{SB}(A)$ the associated Severi-Brauer variety. 
\begin{itemize}
\item[(i)] For every commutative ring $R$ one has the following motivic decomposition
\begin{equation}\label{eq:decomposition-SB}
U(\perf_\dg(X))_R \simeq {\bf 1} \oplus U(\underline{A})_R \oplus U(\underline{A})_R^{\otimes 2} \oplus \cdots \oplus U(\underline{A})_R^{\otimes d-1}\,.
\end{equation}
In particular, the noncommutative motive $U(\perf_\dg(X))_R$ is decomposable.
\item[(ii)] (Karpenko) When $A$ is moreover a division algebra and $d=p^s$ for a certain prime $p$ and integer $s\geq1$, then the Chow motive $M(X)_\bbZ$ (and even $M(X)_{\bbZ/p\bbZ}$) is indecomposable.
\end{itemize}
\end{proposition}
Proposition \ref{prop:main3} shows that the decomposition \eqref{eq:decomposition-SB} is ``truly noncommutative''.

Given a smooth projective $k$-scheme $X$ and an integer $l$, let us write $M(X)_R(l)$ instead of $M(X)_R \otimes {\bf L}^{\otimes l}$. Our third main result, which in particular answers Question C, is the following:
\begin{theorem}\label{thm:new}
Let $\{X_i\}_{1\leq i \leq n}$ (resp. $\{Y_j\}_{1\leq j \leq m}$) be irreducible smooth projective $k$-schemes of dimension $d_i$ (resp. $d_j$), $d:=\mathrm{max}\{d_i,d_j\,|\, i,j\}$, and $\{l_i\}_{1 \leq i \leq n}$ (resp. $\{l_j\}_{1 \leq j \leq m}$) arbitrary integers. Assume that $\bbZ \subseteq R$ and $1/(2d)! \in R$. Under these assumptions, we have the following implication
$$\oplus_i M(X_i)_R(l_i) \simeq \oplus_j M(Y_j)_R(l_j) \Rightarrow \oplus_i U(\perf_\dg(X_i))_R \simeq \oplus_j U(\perf_\dg(Y_j))_R\,.
$$
\end{theorem}
It is unclear to the authors if the (strict) converse of Theorem~\ref{thm:new} also holds. As the following (counter-)example shows, this is false in general !
\begin{example}
The Chow motives $M(X)_\bbZ$ and $M(\widehat{X})_\bbZ$ of an abelian variety $X$ and of its dual $\widehat{X}$ are in general not isomorphic. However, thanks to the work \cite{Mukai}, we have $U(\perf_\dg(X))_R \simeq U(\perf_\dg(\widehat{X}))_R$ for every commutative ring $R$.
\end{example}
Finally, by combining Theorem~\ref{thm:new} with \eqref{eq:categories}, we obtain the application:
\begin{corollary}\label{cor:main2}
Let $X$ (resp. $Y$) be an irreducible smooth projective $k$-scheme of dimension $d_X$ (resp. $d_Y$), and $d:=\mathrm{max}\{d_X, d_Y\}$. Assume that $\bbZ \subseteq R$ and $1/(2d)! \in R$. Under these assumptions, we have the following implication
$$ M(X)_R \simeq M(Y)_R \Rightarrow E(X) \simeq E(Y)$$
for every additive invariant $E$ with values in a $R$-linear category.
\end{corollary}
\section{Preliminaries}\label{sec:preliminaries}
\subsection{Orbit categories}\label{sub:orbit}
Let $\cC$ be an additive symmetric monoidal category and $\cO \in \cC$ a $\otimes$-invertible object. Recall from \cite[\S7]{CvsNC} that the {\em orbit category} $\cC\!/_{\!\!-\otimes \cO}$ has the same objects as $\cC$ and morphisms
$$ \Hom_{\cC\!/_{\!\!-\otimes \cO}}(a,b):= \oplus_{j \in \bbZ} \Hom_\cC(a,b\otimes\cO^{\otimes j})\,.$$
Given objects $a, b$ and $c$ and morphisms
\begin{eqnarray*} \underline{f}=\{f_j\}_{j\in \bbZ} \in \oplus_{j \in \bbZ} \Hom_{\cC}(a,b\otimes \cO^{\otimes j})&\underline{g}=\{g_k\}_{k\in \bbZ} \in \oplus_{k \in \bbZ} \Hom_{\cC}(b,c\otimes \cO^{\otimes k})\,,&
\end{eqnarray*}
the $l^{\mathrm{th}}$-component of the composition $\underline{g} \circ \underline{f}$ is given by the finite sum $\sum_j ((g_{l-j} \otimes \cO^{\otimes j} )\circ f_j)$. We obtain in this way an additive category $\cC\!/_{\!\!-\otimes \cO}$ and a canonical additive projection functor $\pi: \cC \to \cC\!/_{\!\!-\otimes \cO}$. Note that $\pi$ comes equipped with a natural isomorphism $\pi\circ (-\otimes \cO) \stackrel{\sim}{\Rightarrow} \pi$ and that it is universal among all such functors. Note also that this construction is functorial: given any other symmetric monoidal category $\cD$, a $\otimes$-invertible object $\cO'\in \cD$, and an additive $\otimes$-functor $\cC\to \cD$ which sends $\cO$ to $\cO'$, we obtain an induced additive functor $\cC/_{\!\!-\otimes \cO} \to \cD/_{\!\!-\otimes \cO'}$.
\subsection{Dg categories}\label{sub:dg}
Let $\cC(k)$ be the category of complexes of $k$-vector spaces. A {\em differential graded (=dg) category $\cA$} is a category enriched over $\cC(k)$ (morphisms sets $\cA(x,y)$ are complexes) in such a way that composition fulfills the Leibniz rule $d(f \circ g) =d(f) \circ g +(-1)^{\mathrm{deg}(f)}f \circ d(g)$. A {\em dg functor} $F:\cA\to \cB$ is  a functor enriched over $\cC(k)$; consult \cite{ICM}. In what follows we will write $\dgcat(k)$ for the category of (small) dg categories and dg functors. 

Let $\cA$ be a dg category. The {\em opposite} dg category $\cA^\op$ has the same objects as $\cA$ and complexes of morphisms given by $\cA^\op(x,y):=\cA(y,x)$. A {\em right $\cA$-module} $M$ is a dg functor $M:\cA^\op \to \cC_\dg(k)$ with values in the dg category $\cC_\dg(k)$ of complexes of $k$-vector spaces. Let us denote by $\cC(\cA)$ the category of right $\cA$-modules; see \cite[\S2.3]{ICM}. Recall from \cite[\S3.2]{ICM} that the {\em derived category $\cD(\cA)$ of $\cA$} is the localization of $\cC(\cA)$ with respect to the class of objectwise quasi-isomorphisms. A dg functor $F:\cA\to \cB$ is called a {\em Morita equivalence} if the restriction of scalars functor $\cD(\cB) \stackrel{\sim}{\to} \cD(\cA)$ is an equivalence of (triangulated) categories; see \cite[\S4.6]{ICM}. 

The {\em tensor product $\cA\otimes\cB$} of two dg categories $\cA$ and $\cB$ is defined as follows: the set of objects is the cartesian product of the sets of objects of $\cA$ and $\cB$ and the complexes of morphisms are given by $(\cA\otimes\cB)((x,z),(y,w)):= \cA(x,y) \otimes \cB(z,w)$. 

\subsection{$K_0$-motives}\label{sub:K0motives}
Recall from \cite[\S5.2]{GiSou}\cite[\S5.2]{GiSou2} the construction of the category $\KM(k)_R$ of {\em $K_0$-motives}. As explained in {\em loc. cit.}, its objects are the smooth projective $k$-schemes, its morphisms are given by 
$$ \Hom_{\KM(k)_R}(X,Y) := K_0(X\times Y)_R = K_0(X\times Y)_\bbZ \otimes_\bbZ R\,,$$
and its symmetric monoidal structure is induced by the product of $k$-schemes. Furthermore, $\KM(k)_R$ comes equipped with a canonical (contravariant) $\otimes$-functor
\begin{eqnarray}
M_0(-)_R: \SmProj(k)^\op \too \KM(k)_R && X \mapsto X
\end{eqnarray}
that sends a morphism $f:X \to Y$ in $\SmProj(k)$ to the class $[\cO_{\Gamma_f^t}] \in K_0(Y\times X)_R$ of the transpose $\Gamma_f^t$ of the graph $\Gamma_f:= \{ (x,f(x))\,|\, x \in X\}\subset X\times Y$ of $f$.
\begin{notation}\label{not:subcategories}
Given a finite family $X_1, \ldots, X_n$ of irreducible smooth projective $k$-schemes of dimensions $d_1, \ldots, d_n$, let us denote by $(X_1,\ldots, X_n)_R$ the full subcategory of $\KM(k)_R$ consisting of the objects $\{M_0(X_i)_R\,|\, 1 \leq i \leq n\}$. Its closure (inside $\KM(k)_R$) under finite direct sums will be denoted by $(X_1,\ldots, X_n)^\oplus_R$.
\end{notation}
As explained in \cite[\S4.4]{MT}, there is a well-defined $R$-linear additive {\em fully faithful} $\otimes$-functor $\theta$ making the following diagram commute
\begin{equation}\label{eq:diagramK0}
\xymatrix{
\SmProj(k)^\op \ar[d]_-{M_0(-)_R} \ar[rr]^-{\perf_\dg(-)} && \dgcat(k) \ar[d]^-{U(-)_R} \\
\KM(k)_R \ar[rr]_-{\theta} && \Hmo_0(k)_R\,.
}
\end{equation}
Intuitively speaking, the category of $K_0$-motives embeds fully faithfully into the category of noncommutative motives.

\section{Proof of Theorem~\ref{thm:main1}}\label{sec:proof-main1}
Recall that by assumption $\bbZ\subseteq R$. Let us then denote by $R[1/\bbZ]$ the localization of $R$ at $\bbZ\backslash \{0\}$.
\begin{proposition}\label{prop:functor}
There exists a well-defined additive functor $\Psi$ making the following diagram commute
\begin{equation}\label{eq:diagram-first}
\xymatrix{
\SmProj(k)^\op \ar[d]_-{M(-)_{R[1/\bbZ]}} \ar@{=}[r]& \SmProj(k)^\op \ar[dd]^-{M_0(-)_R} \\
\Chow(k)_{R[1/\bbZ]} \ar[d]_-{\pi} & \\
\Chow(k)_{R[1/\bbZ]}/_{\!\!-\otimes R[1/\bbZ](1)} & \KM(k)_R \ar[l]^-{\Psi}\,,
}
\end{equation}
where $R[1/\bbZ](1) \in \Chow(k)_{R[1/\bbZ]}$ denotes the Tate motive.
\end{proposition}
\begin{proof}
Let $X$ and $Y$ be irreducible smooth projective $k$-schemes of dimensions $d_X$ and $d_Y$. As explained in \cite[\S4]{Andre}, given $j \in \bbZ$, one has a canonical isomorphism
$$ \Hom_{\Chow(k)_{R[1/\bbZ]}}(M(X)_{R[1/\bbZ]},M(Y)_{R[1/\bbZ]}\otimes R[1/\bbZ](1)^{\otimes j})\simeq CH^{d_X +j}(X\times Y)_{R[1/\bbZ]}\,,$$
where $CH^{d_X+j}(X\times Y)_{R[1/\bbZ]}$ denotes the $R[1/\bbZ]$-linear Chow group of algebraic cycles of codimension $d_X +j$ on $X \times Y$ modulo rational equivalence. By definition of the orbit category category, the $R[1/\bbZ]$-module
$$
\Hom_{\Chow(k)_{R[1/\bbZ]}/_{\!\!-\otimes R[1/\bbZ](1)}}(\pi(M(X)_{R[1/\bbZ]}),\pi(M(Y)_{R[1/\bbZ]}))
$$
identifies with 
$$\oplus_{j \in \bbZ} CH^{d_X+j}(X\times Y)_{R[1/\bbZ]} = \oplus_{i=0}^{d_X+d_Y} CH^i(X\times Y)_{R[1/\bbZ]}\,.$$
This shows us that the category $\Chow(k)_{R[1/\bbZ]}/_{\!\!-\otimes R[1/\bbZ](1)}$ agrees with the category $\mathbf{CHM}_{kR[1/\bbZ]}$ of {\em all} correspondences considered at \cite[page~3128]{GiSou2}. On the other hand, recall from \S\ref{sub:K0motives} that
$$ \Hom_{\KM(k)_R}(M_0(X)_R, M_0(Y)_R)=K_0(X\times Y)_R\,.$$
The searched functor $\Psi$ is defined on objects by sending $M_0(X)_R$ to $\pi(M(X)_{R[1/\bbZ]})$. On morphisms is defined by the following assignment
\begin{eqnarray*}
K_0(X\times Y)_R \too \oplus_{i=0}^{d_X+d_Y} CH^i(X\times Y)_{R[1/\bbZ]} && \alpha \mapsto ch(\alpha) \cdot \pi^\ast_Y(\mathrm{Td}(Y))\,,
\end{eqnarray*}
where $ch(-)$ denotes the Chern character, $\mathrm{Td}(Y)$ the Todd class of $Y$ and $\pi_Y$ the projection $X\times Y \to Y$ morphism. As explained at \cite[page~3128]{GiSou2}, it follows from the Grothendieck-Riemann-Roch theorem that the above assignments give rise to a well-defined additive functor $\Psi$. The fact that the diagram \eqref{eq:diagram-first} commutes follows also from the Grothendieck-Riemann-Roch theorem; see \cite[page~3129]{GiSou2}.
\end{proof}
Consider the following diagram of additive functors
$$ \Chow(k)_R \to \cdots \to \Chow(k)_{R[1/n!]}\to \Chow(k)_{R[1/(n+1)!]} \to \cdots \to \Chow(k)_{R[1/\bbZ]}\,.$$
Since the Tate motive $R(1)$ is mapped to itself, the functoriality of $(-)/_{\!\!-\otimes R(1)}$ gives rise to the following diagram
$$ \cdots \cdots\cdots \to \Chow(k)_{R[1/n!]}/_{\!\!-\otimes R(1)} \to \cdots \to \Chow(k)_{R[1/\bbZ]}/_{\!\!-\otimes R[1/\bbZ](1)}\,.$$
\begin{proposition}\label{prop:aux-key3}
Given a finite family of irreducible smooth projective $k$-schemes $X_1, \ldots, X_n$ of dimensions $d_1, \ldots, d_n$, the composition (see Notation~\ref{not:subcategories})
$$(X_1,\ldots, X_n)_R \subset \KM(k)_R \stackrel{\Psi}{\too} \Chow(k)_{R[1/\bbZ]}/_{\!\!-\otimes R[1/\bbZ](1)}$$
factors through the functor
$$(\Chow(k)_{R[1/(2d)!]}/_{\!\!-\otimes R(1)})^\natural \too \Chow(k)_{R[1/\bbZ]}/_{\!\!-\otimes R[1/\bbZ](1)}$$
where $d:=\mathrm{max}\{d_1,\ldots, d_n\}$.
\end{proposition}
\begin{proof}
Let $X_r$ and $X_s$ be any two $k$-schemes in $\{X_1, \ldots, X_n\}$. From the construction of the functor $\Psi$ (see the proof of Proposition~\ref{prop:functor}), it is clear that it suffices to show that the homomorphism
\begin{eqnarray*}
K_0(X_r \times X_s)_R \too \oplus_{i=0}^{d_r+d_s} CH^i(X_r \times X_s)_{R[1/\bbZ]} && \alpha \mapsto ch(\alpha)\cdot \pi^\ast_{X_s}(\mathrm{Td}(X_s))
\end{eqnarray*}
factors through
$$ \oplus_{i=0}^{d_r+d_s} CH^i(X_r \times X_s)_{R[1/(2d)!]} \too \oplus_{i=0}^{d_r+d_s} CH^i(X_r \times X_s)_{R[1/\bbZ]}\,.$$
Recall from \cite[\S1]{Pappas} that the Chern character $ch(\alpha)$ is given by $\sum_{m \geq 0} \frac{\mathfrak{S}_m(\alpha)}{m!}$ where $\mathfrak{S}_m(\alpha)$ is a polynomial with integral coefficients in the Chern classes $c_i(\alpha) \in CH^i(X_r \times X_s)_\bbZ$. Moreover, since $X_r \times X_s$ is of dimension $d_r + d_s$, one has $\mathfrak{S}_m(\alpha)=0$ for $m > d_r+d_s$. Hence, in order to show that $\mathrm{ch}(\alpha)$ belongs to $\oplus_{i=0}^{d_r + d_s}CH^i(X_r \times X_s)_{R[1/(2d)!]}$, it suffices to show that the numbers $\{m!\,|\, 0 \leq m \leq d_r+d_s\}$ are invertible in $R[1/(2d)!]$. This follows from Lemma~\ref{lem:aux-arithmetic}(i) below since $d_r + d_s \leq 2d$. Now, recall again from \cite[\S1]{Pappas} that the Todd class $\mathrm{Td}(X_s)$ is given by $\sum_{m \geq 0} \frac{\mathfrak{D}_m}{T_m}$. Here, $\mathfrak{D}_m$ is a polynomial with integral coefficients in the Chern classes and $T_m$ is the product $\prod_p p ^{[\frac{m}{p-1}]}$ taken over all the prime numbers; the 
symbol $[-]$ stands for the integral part. Since $X_s$ is of dimension $d_s$, one has also $\mathfrak{D}_m=0$ for $m > d_s$. Moreover, we have a well-defined ring homomorphism
\begin{equation}\label{eq:ring1}
\pi^\ast_{X_s}: \oplus_{i=0}^{d_s}CH^i(X_s)_{R[1/(2d)!]} \too \oplus_{i=0}^{d_r+d_s}CH^i(X_r\times X_s)_{R[1/(2d)!]}\,.
\end{equation}
Therefore, in order to show that $\mathrm{Td}(X_s)$ belongs to $\oplus_{i=0}^{d_s}CH^i(X_s)_{R[1/(2d)!]}$ (which by \eqref{eq:ring1} implies that $\pi^\ast_{X_s}(\mathrm{Td}(X_s))$ belongs to $\oplus_{i=0}^{d_r+d_s}CH^i(X_r\times X_s)_{R[1/(2d)!]}$), it suffices to show that the numbers $\{T_m\,|\, 0 \leq m \leq d_s\}$ are invertible in $R[1/(2d)!]$. Since $d_s \leq d$ this follows from Lemma~\ref{lem:aux-arithmetic}(ii) below.\end{proof}

\begin{lemma}\label{lem:aux-arithmetic}
Let $d$ be a non-negative integer.
\begin{itemize}
\item[(i)] If $m \leq 2d$ then $m!$ is invertible in $R[1/(2d)!]$. 
\item[(ii)] If $m \leq d$ then $T_m:=\prod_p p ^{[\frac{m}{p-1}]}$ is invertible in $R[1/(2d)!]$.
\end{itemize}
\end{lemma}
\begin{proof}
Item (i) follows simply from the fact that $m!$ is a factor of $(2d)!$. In what concerns item (ii) note first that the case $d=0$ is trivial since $T_0=1$. Let us then assume that $d\geq1$ and $m \geq 1$. Note that $[\frac{m}{p-1}]\neq0$ if and only if $p-1 < m$. Hence, the prime factors of $T_m$ are the prime numbers $p'$ such that $p'\leq m+1 \leq d+1\leq 2d$; in the last inequality we use the assumption $d\geq 1$. By item (i) this then implies that $T_m$ is invertible in $R[1/(2d)!]$.
\end{proof}

We now have all the ingredients needed for the conclusion of the proof of Theorem~\ref{thm:main1}. Recall that by hypothesis $U(\perf_\dg(X))_R\simeq \oplus_{i=1}^m {\bf 1}$ for a certain non-negative integer $m$. Since ${\bf 1}=U(\underline{k})_R \simeq U(\perf_\dg(\mathrm{spec}(k)))_R$ one concludes from the commutativity of diagram \eqref{eq:diagramK0} and from the additiveness and fully faithfulness of $\theta$ that $M_0(X)_R \simeq \oplus_{i=1}^m M_0(\mathrm{spec}(k))_R$. Recall also that by construction the orbit category $\Chow(k)_{R[1/(2d)!]}/_{\!\!-\otimes R(1)}$ is additive. Hence, by extending the functor $\Psi$ to finite direct sums (see Notation~\ref{not:subcategories}) one obtains a well-defined additive functor
\begin{equation}\label{eq:functor-sum}
\Psi^\oplus: (X_1, \ldots, X_n)^\oplus_R \too \Chow(k)_{R[1/(2d)!]}/_{\!\!-\otimes R(1)}\,.
\end{equation}
Note that the above isomorphism $M_0(X)_R \simeq \oplus_{i=1}^m M_0(\mathrm{spec}(k))_R$ belongs to the category $(X,\mathrm{spec}(k))^\oplus$. By applying to it the above functor \eqref{eq:functor-sum} one then obtains
\begin{equation}\label{eq:directsum}
\pi(M(X)_{R[1/(2d)!]})\simeq \oplus_{i=1}^m \pi(M(\mathrm{spec}(k))_{R[1/(2d)!]})\,.
\end{equation}
Consequently, using the equalities
$$
\begin{array}{rcl}
\oplus_{i=1}^m(\pi(M(\mathrm{spec}(k))_{R[1/(2d)!]}))&=&\pi(\oplus_{i=1}^mM(\mathrm{spec}(k))_{R[1/(2d)!]})\\
(\oplus_{i=1}^m M(\mathrm{spec}(k))_{R[1/(2d)!]})\otimes R(j)&=& \oplus_{i=1}^m R(j)\,,
\end{array}
$$
there exist morphisms in the orbit category
$$
\underline{f}=\{f_j\}_{j \in \bbZ}  \in  \oplus_{j\in \bbZ} \Hom_{\Chow(k)_{R[1/(2d)!]}}(M(X),\oplus_{i=1}^m R(j))$$
$$\underline{g}=\{g_k\}_{k \in \bbZ}  \in \oplus_{k\in \bbZ} \Hom_{\Chow(k)_{R[1/(2d)!]}}(\oplus_{i=1}^m M(\mathrm{spec}(k)), M(X) \otimes R(k))$$
verifying the equalities $\underline{g}\circ \underline{f} = \id = \underline{f}\circ \underline{g}$; note that we have removed some subscripts in order to simplify the exposition. As explained in \cite[\S4]{Andre}, one has
$$
 \Hom_{\Chow(k)_{R[1/(2d)!]}}(M(X),\oplus_{i=1}^m R(j))
\simeq  \oplus_{i=1}^m CH^{d+j}(X)_{R[1/(2d)!]}$$
and also the isomorphism
$$\Hom_{\Chow(k)_{R[1/(2d)!]}}(\oplus_{i=1}^m M(\mathrm{spec}(k)),M(X) \otimes R(k))\simeq \oplus_{i=1}^m CH^k(X)_{R[1/(2d)!]}\,.
$$
As a consequence, $f_j=0$ for $j\neq \{-d,\ldots, 0\}$ and $g_k=0$ for $k\neq \{0,\ldots, d\}$. The sets of morphisms $\{f_{-l}\,|\, 0 \leq l \leq d\}$ and $\{g_l \otimes R(1)^{\otimes (-l)}\,|\, 0 \leq l \leq d \}$ give then rise to well-defined morphisms 
\begin{eqnarray*}
\alpha:M(X)_{R[1/(2d)!]} &\too& \oplus_{l=0}^d \oplus _{i=1}^m R(1)^{\otimes (-l)} \\
\beta: \oplus_{l=0}^d \oplus_{i=1}^m R(1)^{\otimes (-l)} &\too& M(X)_{R[1/(2d)!]}
\end{eqnarray*}
in $\Chow(k)_{R[1/(2d)!]}$. The composition $\beta\circ \alpha$ agrees with the $0^{\mathrm{th}}$-component of the composition $\underline{g}\circ \underline{f} = \id$, \ie it agrees with the identity of $M(X)_{R[1/(2d)!]}$. Since $R(1)^{\otimes (-l)}={\bf L}^{\otimes l}$ we conclude then that $M(X)_{R[1/(2d)!]}$ is a direct factor of the Chow motive $\oplus_{l=0}^d \oplus_{i=1}^m {\bf L}^{\otimes l}\in \Chow(k)_{R[1/(2d)!]}$. By definition of the Lefschetz motive ${\bf L}$ we have the following equalities
\begin{equation}\label{eq:Kronecker}
 \Hom_{\Chow(k)_{R[1/(2d)!]}}({\bf L}^{\otimes p}, {\bf L}^{\otimes q}) = \delta_{pq} \cdot R[1/(2d)!] \qquad p,q\geq 0\,,
 \end{equation}
where $\delta_{pq}$ stands for the Kronecker symbol. This implies that $M(X)_{R[1/(2d)!]}$ decomposes into a direct sum (indexed by $l$) of direct factors of $\oplus^m_{i=1} {\bf L}^{\otimes l}$. Note that a direct factor of $\oplus^m_{i=1} {\bf L}^{\otimes l}$ is the same data as an idempotent element of $\End(\oplus^m_{i=1} {\bf L}^{\otimes l})$. Thanks to \eqref{eq:Kronecker} we have the following isomorphism
$$ \End(\oplus^m_{i=1} {\bf L}^{\otimes l}) \simeq M_{m \times m}(R[1/(2d)!])\,,$$
where the right-hand side stands for $m\times m$ matrices with coefficients in $R[1/(2d)!]$. Hence, a direct factor of $\oplus^m_{i=1} {\bf L}^{\otimes l}$ is the same data as an idempotent element of $M_{m \times m}(R[1/(2d)!])$, \ie a finitely projective projective $R[1/(2d)!]$-module. Since by hypothesis all these modules are free we then conclude that the only direct factors of $\oplus^m_{i=1} {\bf L}^{\otimes l}$ are its subsums. As a consequence, $M(X)_{R[1/(2d)!]}$ is isomorphic to a subsum of $\oplus_{l=0}^d \oplus_{i=1}^m {\bf L}^{\otimes l}$ indexed by a subset $S$ of $\{0, \ldots,d\} \times \{1, \ldots, m\}$. By construction of the orbit category we have $\pi({\bf L}^{\otimes l}) \simeq \pi(M(\mathrm{spec}(k))_{\bbZ[1/(2d)!]})$ for every $l \geq 0$. Therefore, since the above direct sum \eqref{eq:directsum} contains $m$ terms we conclude that the cardinality of $S$ is also $m$. This means that there is a choice of integers (up to permutation) $l_1, \ldots, l_m \in \{0, \ldots, d\}$ giving rise to 
an isomorphism
$$ M(X)_{R[1/(2d)!]} \simeq {\bf L}^{\otimes l_1} \oplus \cdots \oplus {\bf L}^{\otimes l_m}$$
in $\Chow(k)_{R[1/(2d)!]}$. This achieves the proof.
\section{Quadratic forms and associated quadrics}\label{sec:quadric}
Recall from \cite{lam-book} the basics about quadratic forms. In this section, $k$ will be a field of characteristic $\neq 2$ and $V$ a finite dimensional $k$-vector space.
\begin{definition}\label{def:quadratic}
Let $(V,q)$ be a quadratic form and $B_q$ the associated bilinear pairing. The dimension of $q$ is by definition the dimension of $V$. 
\begin{itemize}
\item[(i)] The form $(V,q)$ is called {\em non-singular} if the assignment $x \mapsto B_q(-,x)$ gives rise to an isomorphism $V \stackrel{\sim}{\to} V^*$; see \cite[\S I.1]{lam-book}.
\item[(ii)] The form $(V,q)$ is called {\em anisotropic} if the equality $q(x)=0$ holds
 only when $x=0$; see \cite[\S I.3]{lam-book}. Note that when $k$ is algebraically
 closed, any isotropic form has dimension 1.
\item[(iii)] Given two quadratic forms $(V_1,q_1)$ and $(V_2,q_2)$, the {\em orthogonal sum} $(V_1 \oplus V_2,q_1 \perp q_2)$ is the
quadratic form defined by the map $(q_1\perp q_2)(x_1,x_2)= q_1(x_1) +
q_2(x_2)$ \cite[\S I.2]{lam-book}. The {\em tensor product} $(V_1 \otimes V_2, q_1 \otimes q_2)$  is the
quadratic form defined by the map
$(q_1 \otimes q_2)(v_1 \otimes v_2) := q_1(v_1)\cdot q_2(v_2)$; see \cite[\S I.6]{lam-book}.
\item[(iv)] The \it determinant \rm of $q$ is defined as $d(q):= \det(M_q)\cdot(k^*)^2$, where $M_q$
is the matrix of the bilinear form $B_q$ and $k^*$ is the multiplicative group $k \setminus \{ 0 \}$. The determinant of $q$ is then an element
of $k^*/(k^*)^2$ which is well-defined up to isometry; see \cite[\S I.1]{lam-book}. The \it signed determinant \rm
of $q$ is defined as $d_{\pm}(q) := (-1)^{\frac{n(n-1)}{2}}d(q)$, where $n$ is the dimension of $q$; see \cite[\S II.2]{lam-book}.  
\item[(v)] The {\em discriminant extension} $k_q$ defined
by $q$ is the degree 2 quadratic extension $k(\sqrt{d})$ of the base field $k$ (with $d:=d_{\pm}(q)$).
\end{itemize}
\end{definition}
Every quadratic form $q$ gives rise to a $\bbZ/2\bbZ$-graded Clifford algebra $\cC(q)$; see \cite[\S V.1]{lam-book}. 
The even part $\cC_0(q)$ of $\cC(q)$ is called the {\em even Clifford algebra of $q$}.
Suppose $q$ is nonsingular. When $q$ is odd dimensional, $\cC_0(q)$ is a central simple $k$-algebra.
On the other hand, when $q$ is even dimensional, $\cC_0(q)$ is
a central simple $k_q$-algebra, and we have the following two cases: (i) whenever $d_{\pm}(q)$ is not a square in $k$,
the even Clifford algebra $\cC_0(q)$ is a central simple $k_q$-algebra; (ii) whenever $d_{\pm}(q)$ is a square in $k$
(that is, $k_q=k \times k$), the even Clifford algebra $\cC_0(q)$ is the product of two isomorphic central simple $k$-algebras.
{In any case, we get a well defined central simple algebra, i.e. an Azumaya algebra. We denote by $\beta_q$ 
such a central simple algebra and call it the {\em Clifford invariant} of $q$. The following definitions are not standard, but follow automatically from \cite[\S V.2]{lam-book}.

\begin{definition}
Let $(V,q)$ be a non-singular quadratic form over $k$.
\begin{itemize}
\item[(i)] The form $(V,q)$ has {\em trivial discriminant} if $k_q$ splits, \ie if $k_q = k \oplus k$. Equivalently, $(V,q)$ has trivial discriminant if $d_\pm(q)=1 \in k^*/(k^*)^2$.
\item[(ii)] The form $(V,q)$ has \it trivial Clifford invariant \rm if $\beta_q = 0$ 
in the Brauer group.
\end{itemize}
\end{definition}

\begin{remark}\label{remark-splitting-cliff-alg}
An even dimensional quadratic form $q$ has trivial discriminant and trivial Clifford invariant if and only if
$\cC_0(q)=M_r(k) \times M_r(k)$, where $r:=2^{{\mathrm{dim}}(q)-2}$, and $M_r(k)$ denotes the algebra of $r \times r$ matrices over $k$;
see the chart at \cite[page~111]{lam-book}. In particular, $\cC_0$ is Morita equivalent
to $k \times k$.
\end{remark}
As explained in \cite[\S II.1]{lam-book}, the isometry classes of anisotropic quadratic forms over $k$ from the {\em Witt ring} $W(k)$, whose sum (resp. product) is induced by
the orthogonal sum (resp. tensor product) of quadratic forms; see Definition~\ref{def:quadratic}(iii). The classes of the even dimensional anisotropic quadratic forms give rise to the so called {\em fundamental ideal} $I(k) \subset W(k)$; see \cite[\S II]{lam-book}.

If $q$ is a quadratic form whose isometry class lies in $I^3(k)$, then $q$ is anisotropic and has trivial discriminant and trivial Clifford invariant; see \cite[Cor. 3.4]{lam-book}. As proved in \cite[Thm. 6.11]{lam-book}, the converse is also true. Hence, we deduce that a non-singular quadratic form $q$ satisfies the assumptions
of Proposition \ref{prop:counter-example-quadric} if and only if its isometry class belongs to $I^3(k)$. In particular, there
is no such quadratic forms if $I^3(k)= 0$ (e.g. if $k$ is algebraically closed or finite; see \cite[\S XI.6]{lam-book}).

\begin{example}{($3$-fold Pfister forms)}\label{ex:3-fold}
In order to describe quadratic forms in the powers of the fundamental ideal $I(k)$, one considers {\em Pfister forms}.
The isometry class of the 2-dimensional quadratic form $x^2 + ay^2$ is
denoted by $\langle 1, a \rangle$ and called a \it 1-fold Pfister form. \rm  
An \it $n$-fold Pfister form \rm is the isometry class of a $2n$-dimensional
quadratic form $\langle 1,a_1 \rangle \otimes \ldots \otimes \langle 1,a_n \rangle$.
The key property of Pfister forms is that whenever $k$ is a function field, the ideal $I^n(k)$
is additively generated by the $n$-fold Pfister forms; see \cite[\S X.1, Prop. 1.2]{lam-book}.
Hence, whenever $k$ is a function field, $3$-fold Pfister forms satisfy assumptions of Proposition \ref{prop:counter-example-quadric}
\end{example}

\subsection*{Proof of Proposition~\ref{prop:counter-example-quadric}}
The fact that the Chow motive $M(Q_q)_\bbZ$ is {\em not} of Lefschetz $\bbZ$-type was proved in  \cite{rost-motive};
see also \cite[\S XVII]{elman-karpenko-merku}.
In what concerns $U(\perf_\dg(Q_q))_\bbZ$, recall from \cite{auel-bernar-bologn,kuznetquadrics} that we have the following semi-orthogonal decomposition 
$$ \perf(Q_q) = \langle \perf(\cC_0(q)), \cO(-d+1), \ldots, \cO\rangle\,.$$
As proved in \cite[\S5]{MT}, semi-orthogonal decompositions become direct sums
in the category of noncommutative Chow motives. Since $\perf_\dg(\cC_0(q))$
is Morita equivalent to $\underline{\cC_0(q)}$ one then obtains the following motivic decomposition
\begin{equation}\label{eq:iso1}
U(\perf_\dg(Q_q))_\bbZ \simeq U(\underline{\cC_0(q)})_\bbZ \oplus {\bf 1}^{\oplus n}\,.
\end{equation}
Using the above Remark~\ref{remark-splitting-cliff-alg} one has moreover
\begin{equation}\label{eq:iso2}
U(\underline{\cC_0(q)})_\bbZ \oplus {\bf 1}^{\oplus n} \simeq {\bf 1}^{\oplus 2} \oplus {\bf 1}^{\oplus n}\,.
\end{equation} 
By combining \eqref{eq:iso1}-\eqref{eq:iso2} one concludes that $U(\perf_\dg(Q_q))_\bbZ$ is of unit $\bbZ$-type and so the proof is finished.

\section{Proof of Theorem~\ref{thm:main2}}
Let $n \geq 1$. Following \cite[page~498]{Soule}, let us denote by $\cS_n$ the category of those abelian groups $G$ which verify the following two conditions:
\begin{itemize}
\item[(i)] there exist an integer $m$ such that $mg=0$ for all $g \in G$.
\item[(ii)] if $p$ is a prime factor of $m$ then $p=2$ or $p <n$.
\end{itemize}
As explained in {\em loc. cit.}, $\cS_n$ is a Serre subcategory of the category of all abelian groups. We start with the following ``arithmetic'' result:
\begin{lemma}\label{lem:arithmetic}
Given any two abelian groups $G$ and $H$, the following holds:
\begin{itemize}
\item[(i)] Assume that $G=H$ modulo $\cS_1$ or modulo $\cS_2$. Then, $G_{\bbZ[1/2]}\simeq H_{\bbZ[1/2]}$.
\item[(ii)] Assume that $G=H$ modulo $\cS_n$ with $n \geq 3$. Then, $G_{\bbZ[1/(n-1)!]}\simeq H_{\bbZ[1/(n-1)!]}$. 
\end{itemize}
\end{lemma}
\begin{proof}
In the cases where $n=1, 2$ the integer $m$ is always a power of $2$. Hence, if one inverts $2$ one inverts also $m$. In the remaining cases the prime factors of $m$ are always $\leq n-1$. Hence, if one inverts $(n-1)!$ one inverts all the prime factors of $m$ and consequently $m$ itself.
\end{proof}
\begin{proposition}\label{prop:key-Soule}
Let $X$ be an irreducible smooth projective $k$-scheme of dimension $d_X$. Under the assumption $\bbZ\subseteq R$, the following holds:
\begin{itemize}
\item[(i)] The Todd class $\mathrm{Td}(X)$ is invertible in the Chow ring $\oplus_{i=0}^{d_X} CH^i(X)_{R[1/(2d_X)!]}$.
\item[(ii)] The Chern character induces an isomorphism 
\begin{eqnarray}\label{eq:isom-Chern}
K_0(X)_{R[1/(2d_X)!]} \stackrel{\sim}{\too} \oplus_{i=0}^{d_X} CH^i(X)_{R[1/(2d_X)!]} && \alpha \mapsto ch(\alpha)\,.
\end{eqnarray}
\end{itemize}
\end{proposition}
\begin{proof}
Recall from the proof of Proposition~\ref{prop:aux-key3} that the Todd class $\mathrm{Td}(X) \in \oplus_{i=0}^{d_X} CH^i(X)_{R[1/d_X!]}$ is given by $\sum_{m \geq 0} \frac{\mathfrak{D}_m}{T_m}$, where $\mathfrak{D}_m$ is a polynomial with integral coefficients in the Chern classes. Moreover, $\mathfrak{D}_m=0$ for $m > d_X$. Since by definition $\mathfrak{D}_0=T_0=1$ one then observes that $\mathrm{Td}(X)$ is invertible in $\oplus_{i=0}^{d_X}CH^i(X)_{R[1/d_X!]}$ and consequently in $\oplus_{i=0}^{d_X}CH^i(X)_{R[1/(2d_X)!]}$. This proves item (i).

Let us now prove item (ii). The case $d_X=0$ is clear and so we assume that $d\geq 1$. As proved  at \cite[page~52]{Soule}, the Chern character combined with the Gersten-Quillen spectral sequence give rise to the following equality
\begin{eqnarray*}
K_0(X)_\bbZ = \oplus_{i=0}^{d_X} E_2^{i,-i}(X) && \mathrm{modulo}\,\, \cS_{d_X}\,.
\end{eqnarray*}
Moreover, as proved at \cite[Prop.~5.14]{Quillen}, we have the identifications
\begin{eqnarray*}
E_2^{i,-i}(X) \simeq CH^i(X)_\bbZ && 0 \leq i \leq d_X\,.
\end{eqnarray*}
Using Lemma~\ref{lem:arithmetic} one obtains then the following isomorphisms:\begin{eqnarray}
K_0(X)_{\bbZ[1/2]} \simeq \oplus _{i=0}^{d_X} CH^i(X)_{\bbZ[1/2]} && d_X=1,2 \label{eq:iso-11}\\
K_0(X)_{\bbZ[1/(d_X-1)!]} \simeq \oplus _{i=0}^{d_X} CH^i(X)_{\bbZ[1/(d_X-1)!]} && d_X\geq 3 \label{eq:iso-22}\,.
\end{eqnarray}
The searched isomorphisms \eqref{eq:isom-Chern} can now be obtained by tensoring \eqref{eq:iso-11}-\eqref{eq:iso-22} with $R[1/(2d_X)!]$.
\end{proof}
Let $X_1, \ldots, X_n$ be a finite family of irreducible smooth projective $k$-schemes of dimensions $d_1, \ldots, d_n$. Recall from \S\ref{sec:proof-main1} the construction of the functor 
\begin{equation*}
\Psi^\oplus: (X_1, \ldots, X_n)^\oplus_R \too \Chow(k)_{R[1/(2d)!]}/_{\!\!-\otimes R(1)}\,,
\end{equation*}
where $d:=\mathrm{max}\{d_1,\ldots, d_n\}$.
\begin{proposition}\label{prop:induced}
The induced $R[1/(2d)!]$-linear functor
\begin{equation*}
\Psi^\oplus: (X_1,\ldots, X_n)^\oplus_{R[1/(2d)!]} \too \Chow(k)_{R[1/(2d)!]}/_{\!\!-\otimes R(1)}
\end{equation*}
is fully faithful.
\end{proposition}
\begin{proof}
Let $X_r$ and $X_s$ be any two $k$-schemes in $\{X_1, \ldots, X_n\}$. From the construction of $\Psi^\oplus$ it is clear that it suffices to show that the homomorphism 
\begin{equation}\label{eq:induced-homo}
\begin{array}{ccc}
K_0(X_r \times X_s)_{R[1/(2d)!]} &\too& \oplus_{i=0}^{d_r + d_s} CH^i(X_r \times X_s)_{R[1/(2d)!]}\\
\alpha &\mapsto& ch(\alpha) \cdot \pi^\ast_{X_s}(\mathrm{Td}(X_s))
\end{array}
\end{equation}
is an isomorphism. Thanks to Proposition~\ref{prop:key-Soule}(i) (applied to $X=X_s$) the Todd class $\mathrm{Td}(X_s)$ is an invertible element of $\oplus_{i=0}^{d_s}CH^i(X_s)_{R[1/(2d_s)!]}$ and hence of the Chow ring $\oplus_{i=0}^{d_s}CH^i(X_s)_{R[1/(2d)!]}$. Moreover, since
$$ \pi^\ast_{X_s}: \oplus_{i=0}^{d_s}CH^i(X_s)_{R[1/(2d)!]} \too \oplus_{i=0}^{d_r + d_s} CH^i(X_r \times X_s)_{R[1/(2d)!]}$$
is a ring homomorphism we conclude that $\pi^\ast_{X_s}(\mathrm{Td}(X_s))$ is an invertible element of $\oplus_{i=0}^{d_r + d_s} CH^i(X_r \times X_s)_{R[1/(2d)!]}$. Therefore, in order to prove that \eqref{eq:induced-homo} is an isomorphism it suffices to show that the induced Chern character homomorphism 
\begin{eqnarray*}
&K_0(X_r \times X_s)_{R[1/(2d)!]} \too \oplus_{i=0}^{d_r + d_s}CH^i(X_r \times X_s)_{R[1/(2d)!]} & \alpha \mapsto ch(\alpha)
\end{eqnarray*}
is an isomorphism. This follows now from Proposition~\ref{prop:key-Soule}(ii) above applied to $X=X_r \times X_s$.
\end{proof}

We now have all the ingredients needed for the conclusion of the proof of Theorem~\ref{thm:main2}. Recall that by hypothesis $X$ is an irreducible smooth projective $k$-scheme of dimension $d$. By combing the commutativity of diagram \eqref{eq:diagramK0} with the fully faithfulness of the functor $\theta$ one obtains an $R[1/(2d)!]$-algebra isomorphism 
\begin{equation*}
\End(U(\perf_\dg(X))_{R[1/(2d)!]}) \simeq \End(M_0(X)_{R[1/(2d)!]})\,.
\end{equation*}
Thanks to the fully faithfulness of the functor $\Psi^\oplus$ of Proposition~\ref{prop:induced} (applied to the category $(X)^\oplus_{R[1/(2d)!]}$) and the commutativity of diagram \eqref{eq:diagram-first}, one has moreover
\begin{equation*}
\End(M_0(X)_{R[1/(2d)!]}) \simeq \End(\Psi^\oplus(M_0(X)_{R[1/(2d)!]}))\simeq \End(\pi(M(X)_{R[1/(2d)!]}))\,.
\end{equation*}
Now, recall that by construction the projection functor 
$$ \pi:\Chow(k)_{R[1/(2d)!]} \too \Chow(k)_{R[1/(2d)!]}/_{\!\!-\otimes R(1)}$$
is faithful. Consequently, one obtains the following inclusion of $R[1/(2d)!]$-algebras
\begin{equation*}
\End(M(X)_{R[1/(2d)!]}) \hookrightarrow \End(U(\perf_\dg(X))_{R[1/(2d)!]})\,.
\end{equation*}
This automatically gives rise to the searched implication \eqref{eq:implication-indecomp}.
\section{Proof of Proposition~\ref{prop:main3}}
Item (ii) was proved in \cite[Thm.~2.2.1]{karpenko-brauer-severi} for $M(X)_{\bbZ}$.
The same result holds for $M(X)_{\bbZ/p\bbZ}$; see \cite[Cor. 2.22]{karpenko-incompressibility}. Let us now show item (i). Recall from \cite{marcellobrauer} that we have the following semi-orthogonal decomposition
$$ \perf(X) = \langle\perf(k), \perf(A), \perf(A^{\otimes 2}), \ldots, \perf(A^{\otimes d-1}) \rangle\,.$$
As proved in \cite[\S5]{MT}, semi-orthogonal decompositions become direct sums in the category of noncommutative motives. Since $\perf_\dg(A^{\otimes i})$ is Morita equivalent to $\underline{A}^{\otimes i}$, one then obtains the following motivic decomposition 
\begin{equation}\label{eq:decomp-SB1}
U(\perf_\dg(X))_R \simeq U(\underline{k})_R \oplus U(\underline{A})_R \oplus U(\underline{A}^{\otimes 2})_R \oplus \ldots \oplus U(\underline{A}^{\otimes d-1})_R\,.
\end{equation}  
Finally, since the functor $U(-)_R$ is symmetric monoidal, \eqref{eq:decomp-SB1} identifies with \eqref{eq:decomposition-SB}. 
\begin{remark}
Item (ii) holds also for $M(X)_{\bbZ/p^n\bbZ}$ and hence on the $p$-adic integers; see \cite[Rmq. 2.3]{declercq}.
\end{remark}

\section{Proof of Theorem \ref{thm:new}}
Note first that, by combining the commutativity of diagram \eqref{eq:diagramK0} with the fully faithfulness of the functor $\theta$, it suffices to prove the implication
\begin{equation}\label{eq:implication-11}
\oplus_{i=1}^n M(X_i)_R(l_i) \simeq \oplus_{j=1}^m M(Y_j)_R(l_j) \Rightarrow \oplus_{i=1}^n M_0(X_i)_R \simeq \oplus_{j=1}^m M_0(Y_j)_R\,.
\end{equation}
As explained in \S\ref{sub:orbit}, the projection functor $\pi: \Chow(k)_R \to \Chow(k)_R/_{-\otimes R(1)}$ is additive and moreover sends $M(X_i)_R(l_i)$ to $\pi(M(X_i)_R)$ (up to isomorphism). Hence, the left-hand side of \eqref{eq:implication-11} gives rise to an isomorphism
\begin{equation}\label{eq:iso-last}
\oplus_{i=1}^n \pi(M(X_i)_R) \simeq \oplus_{j=1}^m \pi(M(Y_j)_R)\,.
\end{equation}
Since by hypothesis $1/(2d)! \in R$, Proposition~\ref{prop:induced} furnish us a fully faithful functor
$$ \Psi^\oplus: (X_1, \ldots, X_n,Y_1, \ldots, Y_m)^\oplus_R \too \Chow(k)_R/_{-\otimes R(1)}\,.$$
Using the commutativity of diagram \eqref{eq:diagram-first}, one observes that 
\begin{eqnarray*}
& \Psi^\oplus(\oplus_{i=1}^n M_0(X_i)_R) \simeq \oplus_{i=1}^n \pi(M(X_i))_R & \Psi^\oplus(\oplus_{j=1}^m M_0(Y_j)_R) \simeq \oplus_{j=1}^m \pi(M(Y_j))_R \,.
\end{eqnarray*}
By combining these isomorphisms with \eqref{eq:iso-last}, one obtains then the right-hand side of \eqref{eq:implication-11}. This achieves the proof.

\medbreak\noindent\textbf{Acknowledgments:} The authors are grateful to Asher Auel, Charles De Clercq, Henri Gillet, Sergey Gorchinskiy, Alexander Merkurjev, Christophe Soul{\'e}, Burt Totaro and Chuck Weibel for useful discussions and e-mail exchanges.

\end{document}